\documentclass{amsart}
\usepackage{amsmath}
\usepackage{amssymb}
\usepackage{array}
\usepackage{amscd}

\numberwithin{equation}{section}
\newtheorem{Th}[subsection]{Theorem}
\newtheorem*{Th*}{Theorem}
\newtheorem{Lemma}[subsection]{Lemma}

\newtheorem{Conjecture}[subsection]{Conjecture}

\theoremstyle{definition}
\newtheorem{definition}[subsection]{Definition}
\newtheorem*{definition*}{Definition}
\newtheorem{Remark}[subsection]{Remark}
\newtheorem{Example}[subsection]{Example}
\newcommand{\comm}[1]{}

\usepackage{color}
\usepackage[normalem]{ulem}
\definecolor{DarkGreen}{rgb}{0,0.5,0.1} 

\newcommand\soutD{\bgroup\markoverwith
{\textcolor{DarkGreen}{\rule[.5ex]{2pt}{1pt}}}\ULon}

\begin{document}

\title{Smooth \emph{\MakeLowercase{l}}-Fano weighted complete intersections}
\author{Anastasia V.~Vikulova}
\address{{\sloppy
\parbox{0.9\textwidth}{
Steklov Mathematical Institute of Russian
Academy of Sciences,
8 Gubkin str., Moscow, 119991, Russia
\\[5pt]
Laboratory of Algebraic Geometry, National Research University Higher
School of Economics, 6 Usacheva str., Moscow, 119048, Russia.
}\bigskip}}
\email{vikulovaav@gmail.com}

\begin{abstract}
In this paper we prove that for $n$-dimensional smooth \mbox{$l$-Fano} well-formed weighted complete intersections that are not isomorphic to a usual projective space, the upper bound for $l$ is equal to~\mbox{$\lceil \log_2(n+2) \rceil-1 .$}  We also prove that the only $l$-Fano variety of dimension $n$ among such manifolds, where~\mbox{$ \lceil \log_3(n+2) \rceil \leqslant l \leqslant \lceil \log_2(n+2) \rceil -1 $}, is a complete intersection of quadrics in a usual projective space.  

\vspace{11pt}
	
	
	\vspace{11pt}


\end{abstract}

\maketitle

\tableofcontents

	
	\section{Introduction}
	
A smooth projective variety $X$ is called a Fano variety, if the first Chern class of the tangent bundle $\text{c}_1(X)$ is ample. 
Varieties which are called $l$-Fano varieties and which are expected to satisfy the property of higher rational connectedness, than  Fano varieties, were introduced in the papers~\cite{l-Fano} and~\cite{Starr}. 
In this paper, we always work over the field of complex numbers.

\begin{definition} 
We say that the Chern character $\text{ch}_i(X)$ of a smooth projective variety $X$  is positive if~\mbox{$\text{ch}_i(X) \cdot Z >0$} for every effective $i$-cycle $Z.$
\end{definition}

\begin{definition}\label{Defl-Fano}
A smooth Fano variety $X$ is called an \emph{l-Fano} variety if the Chern characters~$\text{ch}_i(X)$ are positive for all $2 \leqslant i \leqslant l.$
\end{definition}

\begin{Example}
It is obvious that a usual projective space~$\mathbb{P}^n$ is an $l$-Fano variety for $l=1, \ldots, n$. Indeed, we have~$\text{ch}_k(\mathbb{P}^n)\,=\,\frac{n+1}{k!}H^k,$ where $H$ is a hyperplane class in~$\mathbb{P}^n.$ 
\end{Example}

\begin{Example}\label{exquadric}
If $X \subset \mathbb{P}^N$ is a smooth complete intersection of hypersurfaces of degrees $d_1, \ldots , d_k,$ then $X$ is an \mbox{$l$-Fano} variety if and only if 
$$
\sum_{i=1}^k d_i^l < N+1
$$
\noindent (see, for example,~\mbox{\cite[Section 2, Example 3]{Starranote}).} In particular, we immediately get that a smooth complete intersection of~$k$ quadrics $X$ in the projective space~$\mathbb{P}^N$ is \mbox{$l$-Fano} if and only if
$$
l \leqslant \left\lceil \log_2\left(\frac{N+1}{k}\right) \right\rceil -1.
$$
\noindent  In other words, a complete intersection of $k$ quadrics in $\mathbb{P}^N$ is an $l$-Fano variety if and only if 
$$
k \leqslant \left\lceil \frac{N+1}{2^l}\right\rceil-1.
$$
 In the same way we get that a smooth cubic hypersurface in $\mathbb{P}^N$ is \mbox{$l$-Fano} if and only if
$$
l \leqslant \lceil \log_3(N+1) \rceil-1.
$$
\end{Example}

The study of $l$-Fano varieties is motivated by the following observations. While a general point of a Fano variety is contained in a rational curve (see~\cite[Chapter~5, Theorem~1.6.1]{Kollar}),  a general point of a $2$-Fano variety is contained in a rational surface under some mild assumptions (see~\cite{Starr}). In the paper~\cite{Castravet2012} the same statement was proved under weaker assumptions, and, moreover, it was proved that a general point of a \mbox{$3$-Fano} variety is contained in a rational threefold under some additional assumptions.  Another property of Fano varieties is outlined in Tsen theorem which states that a smooth Fano hypersurface $X \subset \mathbb{P}^N_K$ of degree $d<N+1$  over a field $K/k$ of transcendence degree~$1$ over an algebraically closed field $k$ always has a $K$-point. At the same time, a similar property of $l$-Fano varieties is outlined in Tsen--Lang theorem which states that  a smooth \mbox{$l$-Fano} hypersurface $X \subset \mathbb{P}^N_K$ of degree $d$ with~\mbox{$d^l<N+1$}  over a field $K/k$ of transcendence degree $l$ over an algebraically closed field $k$ always has a $K$-point  (see \cite[Theorem~6]{Lang}).

The following conjecture was suggested in~\cite{l-Fano}:

\begin{Conjecture}[{\cite[Conjecture 1.7]{l-Fano}}]\label{conjecture}
If $X$ is an n-dimensional smooth \mbox{$l$-Fano} variety with $l \geqslant \lceil \log_2(n+2) \rceil,$ then $X \simeq \mathbb{P}^n.$
\end{Conjecture}

There has been some progress in proving this conjecture. In~\cite[Theorems~1.3 and 1.5]{l-Fano}  a classification of $2$-Fano varieties among rational homogeneous spaces of Picard rank $1$ and a classification of $3$-Fano varieties of dimension~$n$ and index at least $n-2$ were obtained (here the index is the maximum positive integer which divides the canonical class), and it was shown that  Conjecture~\ref{conjecture} is true for these two classes of varieties. In the paper~\cite{Sato20}  it was shown by computer calculations that the only smooth toric varieties of dimension at most $8$ with positive second Chern character are projective spaces. Thus, for toric varieties of dimension at most $8$ Conjecture~\ref{conjecture} is true.

The purpose of this paper is to prove Conjecture \ref{conjecture} in the case of smooth weighted complete intersections (see \cite{IF, PS26} or Section~\ref{Sectofdefinitions} below for definitions and basic properties of these varieties).

\begin{Remark}
 Note that we assume that codimension of a smooth weighted complete intersection is positive, that is, a weighted projective space is not considered as a weighted complete intersection in itself. Note also that by~\cite[Theorem 2.7]{PS20} a smooth well-formed weighted complete intersection which is not an intersection with a linear cone is not isomorphic to a projective space.
\end{Remark}

 We are going to prove the following theorem:

\begin{Th}\label{log_2}
Let  $X \subset \mathbb{P}(a_0,\ldots ,a_N)$ be a smooth $n$-dimensional well-formed weighted complete intersection which is not an intersection with a linear cone.
Suppose that $X$ is $l$-Fano. Then  \mbox{$l< \lceil \log_2(n+2) \rceil.$}
\end{Th}

In other words, Theorem \ref{log_2} states that Conjecture \ref{conjecture} holds for smooth well-formed weighted complete intersections which are not intersections with a linear cone. Moreover, in this case the following strengthening of Theorem~\ref{log_2} holds.

\begin{Th}\label{log_3}
Let $X \subset\mathbb{P}= \mathbb{P}(a_0, \ldots ,a_N)$ be a smooth n-dimensional well-formed weighted complete intersection  which is not an intersection with a linear cone.
Suppose that $X$ is $l$-Fano, and 
$$
 \lceil \log_3(n+2) \rceil \leqslant l < \lceil \log_2(n+2) \rceil. 
$$
\noindent Then 
$$
\mathbb{P}=\mathbb{P}(1, \dots,1)=\mathbb{P}^N
$$
\noindent and $X \subset \mathbb{P}$ is a complete intersection of at most~\mbox{$\lceil \frac{N+1}{2^l} \rceil-1$} quadrics.
\end{Th}

\begin{Remark}
By~\mbox{\cite[Section 2, Example 3]{Starranote}} a smooth weighted complete intersection 
$X \subset  \mathbb{P}(a_0, \ldots, a_N)$ of multidegree $(d_1, \ldots, d_k)$ is $l$-Fano if and only if
\begin{equation}\label{fanoinequalities}
\sum_{j=0}^N a_j^m > \sum_{i=1}^k d_i^m \;\; \text{for all} \;\; m=1, \ldots ,l. 
\end{equation}

\noindent In Section \ref{section:cherncharacters} we show that in order to check the $l$-Fano condition for~$X$ it is sufficient to show that  inequality~(\ref{fanoinequalities}) holds for $m=l.$
\end{Remark}

The plan of the paper is as follows. 
In Section~\ref{Sectofdefinitions} we collect some basic facts about weighted complete intersections. 
In Section \ref{section:cherncharacters} we study the Chern characters of Fano varieties and prove Theorem  \ref{monotonic} about monotonicity property of smooth \mbox{$l$-Fano} varieties; namely we show that if the $(l+1)$st Chern character of a smooth well-formed weighted complete intersection which is not an intersection with a linear cone is positive, then its $l$th character is positive as well. 
 In Section \ref{prooflog_2} we prove Theorem~\ref{log_2}.
 In Section~\ref{prooflog_3} we prove Theorem~\ref{log_3}. 
In Sections~\mbox{\ref{append} and \ref{appendfortheoremlog_3}}  we perform some auxiliary elementary computations.

\vspace{5mm}

\textbf{Acknowledgment.} The author wishes to warmly thank  C.~A.~Shramov for proposing this problem to her, his constant attention to this paper, his lessons
in writing mathematics papers and for important remarks. The author is grateful to A.~S.~Trepalin for reading the manuscript and his help with the simplification of some proofs. She also thanks  A.~G.~Kuznetsov for interesting discussions and the anonymous referee for attentive reading of this text.

 The author is a winner of the all-Russia mathematical August Moebius contest of graduate and undergraduate student papers and thanks the jury and the board for the high praise of here work.  The author was partially supported by Theoretical Physics and Mathematics Advancement Foundation “BASIS” and by the HSE University Basic Research Program.

\section{Preliminaries}\label{Sectofdefinitions}

First of all, we recall some definitions.

\begin{definition}
The variety $\mathbb{P}(a_0,\ldots , a_N)=\text{Proj}\,\mathbb{C}[x_0,\ldots ,x_N],$ where the weight of the variable $x_i$ is equal to $a_i,$ is called a weighted projective space. 
\end{definition} 

\begin{definition}
A weighted complete intersection  $X \subset \mathbb{P}(a_0,\ldots ,a_N)$ of multidegree $(d_1, \ldots , d_k),$ where $k \geqslant 1,$ is an intersection of  hypersurfaces of degrees~\mbox{$d_1, \ldots ,d_k$} such that the codimension of every irreducible component of $X$ is equal to $k.$ 
\end{definition}

\begin{definition}\label{wellformed}
We say that a weighted projective space $\mathbb{P}(a_0,\ldots ,a_N)$ is well-formed if $\gcd(a_0,\ldots ,\widehat{a_i},\ldots ,a_N)\,=\,1$ for all $i.$ We say that a weighted complete intersection $X$ in a well-formed projective space $\mathbb{P}(a_0,\ldots ,a_N)$ is well-formed if 
$$
\text{codim}_{X}(X \cap \text{Sing}\, \mathbb{P}(a_0,\ldots ,a_N)) \geqslant 2.
$$
\end{definition}

\begin{definition}\label{intersectioncone}
Let $X \subset \mathbb{P}(a_0,\ldots ,a_N)$ be a weighted complete intersection of multidegree~\mbox{($d_1, \ldots ,d_k)$} in a well-formed weighted projective space. We say that~$X$ is not an intersection with a linear cone, if  $d_i \neq a_j$ for all $i$ and $j.$
\end{definition}

\begin{Remark}
It is reasonable to consider $X$ such that it is not an intersection with a linear cone because of the following observation. Assume that $X \subset \mathbb{P}(a_0,\ldots ,a_N)$ is a degree~\mbox{$d$} hypersurface which is an intersection with a linear cone, that is, let~\mbox{$d=a_i$} for some $i.$ Then the general hypersurface is the zero set of a polynomial~\mbox{$f=x_i+g,$} where $g$ is a polynomial of degree~$d$ of variables $x_j$ of weight $a_j$ for~\mbox{$j=0, \ldots, N.$}  Such hypersurfaces are parametrized by a Zariski open subset of the space of degree $d$ hypersurfaces (see \cite[Chapter 6]{IF}). Thus, if~\mbox{$X \subset \mathbb{P}(a_0,\ldots ,a_N)$} is a general smooth complete intersection of multidegree~\mbox{$(d_1, \ldots, d_k)$} and $d_r=a_s$ for some $r$ and $s,$ then $X \simeq X' \subset \mathbb{P}(a_0,\ldots ,\widehat{a_s},\ldots ,a_N),$ where $X'$ is a smooth weighted complete intersection of multidegree~\mbox{$(d_1,\ldots ,\widehat{d_r},\ldots ,d_k).$}
\end{Remark}

The main auxiliary result used in the proofs of Theorems \ref{log_2}, \ref{log_3} and \ref{monotonic} is as follows: 

\begin{Th}\label{conditions}
Let $X \subset \mathbb{P}(a_0,\ldots ,a_N)$ be a smooth well-formed Fano weighted complete intersection of multidegree~\mbox{$(d_1,\ldots ,d_k)$} which is not an intersection with a linear cone. Assume that
$$
d_1 \leqslant d_{2} \leqslant \ldots  \leqslant d_k \qquad \text{and} \qquad a_0 \leqslant a_{1} \leqslant \ldots  \leqslant a_N.
$$ 

\noindent Then 

\begin{enumerate}
\renewcommand\labelenumi{\rm (\arabic{enumi})}
\renewcommand\theenumi{\rm (\arabic{enumi})}

\item\label{2}
$
d_i \,>\, a_{N-k+i} \quad \text{for} \quad 1 \leqslant i \leqslant k-1;
$

\item\label{5}
$
\#\{j \mid a_j\,=\,1\} \geqslant k+1;
$

\item\label{6}
$
\#\{j \mid a_j\,=\,1\} \geqslant \sum_{j=0}^N a_j-\sum_{i=1}^k d_i;
$

\item\label{3}
$
d_i \neq a_j \;\; \text{for all} \;\; i \; \text{and} \; j;
$

\item\label{3'}
$
d_i \neq 1 \;\; \text{for all} \;\; i;
$

\item\label{4}
$
N \geqslant 2k;
$

\item\label{7}
$
\#\{j \mid b \;\; \text{divides} \;\; a_j\} \leqslant \#\{i \mid b \;\; \text{divides} \;\; d_i\} \;\; \text{for all} \;\; b \in \mathbb{N}_{>1};
$

\item\label{8}
$
d_1\cdot \ldots \cdot d_k \;\; \text{is divisible by} \;\; a_0 \cdot \ldots \cdot a_N;
$

\item\label{1}
$
d_k \geqslant 2a_N.
$

\end{enumerate}

\end{Th}

\begin{proof}
 We refer to~\cite[Lemma 18.14]{IF} for \ref{2};  we refer to~\mbox{\cite[Corollaries 5.3 and 5.11]{Tasin}} for \ref{5} and \ref{6}; assertions \ref{3} and \ref{3'} follow from assertion \ref{5} and from the assumption that $X$ is not an intersection with a linear cone; for \ref{4} we refer to~\mbox{\cite[Theorem 1.3]{CCC}}; for \ref{7} we refer to~\mbox{\cite[Lemma 2.15]{PS16}} and \ref{8} follows from~\ref{7}.

Assertion \ref{1} can be proved using \ref{3'} and \ref{7}. Indeed, if all the $a_j$ are equal to~$1,$ then by~\ref{3'} we have 
$$
d_k \geqslant 2=2a_N.
$$
\noindent If $a_N \neq 1,$ then according to \ref{7}, there is $d_t$ such that $d_t=m a_N$ for some integer~\mbox{$m>1.$} Thus, we obtain 
$$
d_k \geqslant d_t=m a_N \geqslant 2 a_N.
$$  
 
\end{proof}

\section{Chern characters}\label{section:cherncharacters}

Let $X \subset \mathbb{P}=\mathbb{P}(a_0,\ldots ,a_N)$ be a smooth well-formed weighted complete intersection of multidegree $(d_1, \ldots, d_k).$ It is well known (for example, \mbox{see~\cite[Section 6]{l-Fano}}) that the Chern characters can be represented in a simple way as
\begin{equation}\label{eq:chern}
\text{ch}_l(X)\,=\,\frac{1}{l!}(\sum_{j=0}^N a_j^l-\sum_{i=1}^k d_i^l)H^l\vert_X,
\end{equation} 

\noindent where $H$ is a generator of the class group $\text{Cl}(\mathbb{P}).$ This follows immediately from the exact sequence
\begin{equation}
0 \to T_X \to T_{\mathbb{P}}\vert_X \to \bigoplus_{j=1}^k \mathcal{O}_{\mathbb{P}}(d_j)\vert_X \to 0,\label{exactsequence}
\end{equation}

\noindent where $T_{\mathbb{P}}=\Omega_{\mathbb{P}}^\vee$ is the tangent sheaf.  Let us observe that (\ref{exactsequence}) is well defined, because we assume that $X$ is well-formed and thus  it does not pass through the singular locus of $\mathbb{P}$ by~\mbox{\cite[Proposition 2.11]{PS16}}. The exact sequence  
$$
0 \to \mathcal{O}_{\mathbb{P}} \to \bigoplus_{j=0}^N \mathcal{O}_{\mathbb{P}}(a_j) \to T_{\mathbb{P}} \to 0, \notag
$$

\noindent (for example, see~\cite[Theorem 8.1.6]{Cox})  gives us the Chern characters of $T_{\mathbb{P}}.$

We  also prove the following theorem for Chern characters:

\begin{Th}\label{monotonic}
Let $X \subset \mathbb{P}(a_0,\ldots ,a_N)$ be a smooth well-formed Fano weighted complete intersection which is not an intersection with a linear cone. 
Then~$X$ is \mbox{$l$-Fano} if and only if~$\emph{\text{ch}}_l(X)$ is positive. 
\end{Th}

\begin{proof}
We should prove that if $\text{ch}_l(X)$ is positive, then all $\text{ch}_m(X)$ for \mbox{$2 \leqslant m \leqslant l-1$} are positive as well. So we have to prove that $\text{ch}_m(X) \cdot Z >0$ for every effective \mbox{$m$-cycle~$Z$} and all $2 \leqslant m \leqslant l-1.$ But  we know by (\ref{eq:chern}) that $\text{ch}_m(X)$ is a multiple of~$H^m\vert_X$ which is a power of an ample divisor. Thus, we have  $H^m\vert_X \cdot Z >0$ for every effective $m$-cycle $Z.$ So we have to prove that for $m \leqslant l$ the Chern character~$\text{ch}_m(X)$ is equal to 
 $H^m\vert_X$ times a positive number. To do this we prove that
\begin{equation}\label{eq:cherninequality}
\sum_{j=0}^N a_j^m-\sum_{i=1}^k d_i^m>\sum_{j=0}^N a_j^{m+1}-\sum_{i=1}^k d_i^{m+1}
\end{equation}

\noindent for all integers $m \geqslant 1.$ 

From Theorem~\ref{conditions}\ref{6} and Fano condition we get that if 
$$
s\,=\,\#\{j \mid a_j\,=\,1\},
$$
\noindent then 
$$
\sum_{j=s}^N a_j\, \leqslant \, \sum_{i=1}^k d_i \,<\,\sum_{j=s}^N a_j+s.
$$

\noindent As all assumptions of Lemma \ref{f'(l)>0} hold in the theorem, thus we see that inequality~(\ref{eq:cherninequality}) is true.

\end{proof}

It is not true in general that if the $l$th Chern character of a smooth Fano variety is positive, then the ($l-1$)st one is positive as well. Let us consider the following example.

\begin{Example}
Let $X$ be a blow-up of a point on the projective space $\mathbb{P}^n.$ Its Picard group is~\mbox{$\text{Pic}(X)=\mathbb{Z} H \oplus \mathbb{Z}E,$} where $H$ is the pullback of the hyperplane class in~$\mathbb{P}^n$ and~$E$ is an exceptional divisor. We have~\mbox{$X \simeq \mathbb{P}(\mathcal{O}_{\mathbb{P}^{n-1}} \oplus \mathcal{O}_{\mathbb{P}^{n-1}}(-1))$} and there is a natural \mbox{$\mathbb{P}^1$-fibration} $\pi: X \to \mathbb{P}^{n-1}.$ If $H'$ is a hyperplane class in~$\mathbb{P}^{n-1},$ then~\mbox{$\pi^*(H')=H-E.$} 

According to~\cite[Lemma 2.3]{Coskun} the cone of effective $k$-cycles of $X$ for $1 \leqslant k<n$ is generated by
\begin{itemize}
\item $X_k=H^{n-k}-(-1)^{n-k+1}E^{n-k},$ and
\item $Y_k=(-1)^{n-k+1}E^{n-k}.$
\end{itemize}

\noindent Also note that the following relations on $H$ and $E$ hold:
\begin{equation}\label{intersectionHE}
H \cdot E=0, \;\; H^n=1 \;\; \text{and} \;\; E^n=(-1)^{n+1}. 
\end{equation}

\noindent We have the exact sequence

$$
0 \to L \to T_X \stackrel{d\pi}{\longrightarrow} \pi^*T_{\mathbb{P}^{n-1}} \to 0,
$$
\noindent where $L \simeq K_{X}^{-1} \otimes \pi^*K_{\mathbb{P}^{n-1}}$ is an invertible sheaf on $X.$ So  the Chern character of~$X$ is
$$
\text{ch}(X)=\text{ch}(K_X^{-1})\text{ch}(\pi^*K_{\mathbb{P}^{n-1}})+\text{ch}(\pi^*T_{\mathbb{P}^{n-1}}).
$$

\noindent By definition we can rewrite this as
$$
\text{ch}(X)=\exp(H+E)+n\exp(H-E)-1.
$$

\noindent From (\ref{intersectionHE}) we obtain
$$
\text{ch}_k(X)=\frac{(H+E)^k}{k!}+n\frac{(H-E)^k}{k!}=\frac{1}{k!}((n+1)H^k+(-1)^k(n+(-1)^k)E^k).
$$

\noindent So we can calculate the intersection of $\text{ch}_k$ with effective $k$-cycles for $k<n$:
$$
\text{ch}_k(X) \cdot X_k=\frac{1}{k!}\left(1-(-1)^k\right)
$$
\noindent and
$$
\text{ch}_k(X) \cdot Y_k=\frac{1}{k!}\left(n+(-1)^k\right).
$$

\noindent For $k=n$ we obtain
$$
\text{ch}_n(X)=\frac{1}{n!}\left(1-(-1)^n\right).
$$

\noindent Thus  we get that $\text{ch}_k(X)$ is positive for odd $k$ and nef but not positive for even $k.$ In particular, $X$ is not a $2$-Fano variety.
\end{Example}

\begin{Remark}
There are other examples of varieties with the same behaviour of the Chern characters (see, for example,~\mbox{\cite[Example 4.1]{Sato19}}). There is a conjecture  that all smooth toric varieties except $\mathbb{P}^n$ are not  $2$-Fano. See, for example, the paper~\cite{Sato19}. However, if we consider $\mathbb{Q}$-factorial terminal varieties, then there is an example of a toric $2$-Fano variety which is not isomorphic to $\mathbb{P}^n$ (see~\cite{Sato18}). 
\end{Remark}

\section{Proof of Theorem \ref{log_2}}\label{prooflog_2}

 Without loss of generality we can assume that
$$
d_1 \leqslant d_{2} \leqslant \ldots  \leqslant d_k \qquad \text{and} \qquad a_0 \leqslant a_{1} \leqslant \ldots  \leqslant a_N.
$$
\noindent We can also assume that $X$ is a Fano variety, because otherwise there is nothing to prove. It is sufficient to prove that for $l = \lceil \log_2(N-k+2) \rceil $ \mbox{the $l$th Chern} character is nonpositive, that is,
$$
\sum_{i=1}^k d_i^l \geqslant \sum_{j=0}^N a_j^l.
$$

\noindent By Lemma~\ref{f'(l)>0} it is sufficient to prove this for
$$
l =\log_2(N-k+2).
$$ 
Assume that~\mbox{$\sum_{i=1}^k d_i^l < \sum_{j=0}^N a_j^l.$} Taking into account assertions \ref{5}, \ref{3'} and~\ref{1} of Theorem~\ref{conditions} and the fact that $a_N \geqslant a_j$ for $k+1 \leqslant j \leqslant N,$ we obtain
$$
k-1+2^la_N^l \,<\, \sum_{i=1}^k d_i^l \,<\, \sum_{j=0}^N a_j^l \leqslant k+1+a_N^l(N-k),
$$ 
\noindent  that is,
$$
k-1+(N-k+2)a_N^l\,<\,k+1+a_N^l(N-k),
$$

\noindent so that $a_N^l<1,$ which is impossible. 

Theorem~\ref{log_2} is proved.

\section{Proof of Theorem \ref{log_3}}\label{prooflog_3}

Let us prove some preliminary lemmas.

\begin{Lemma}\label{lemmaaboutconic}
Let $X$ be a smooth well-formed Fano weighted complete intersection of codimension $k$ in $\mathbb{P}(a_0, \ldots, a_N),$ which is not an intersection with a linear cone. Assume that~\mbox{$\lceil \log_3(N-k+2) \rceil=1$}. Then $N=2,$ $k=1$ and $X$ is a conic in~\mbox{$\mathbb{P}(1,1,1)=\mathbb{P}^2.$}
\end{Lemma}
\begin{proof}
As $\lceil \log_3(N-k+2) \rceil=1,$ we have $N-k+2 \leqslant 3,$ or equivalently,~\mbox{$N-k \leqslant 1.$} The case when $N-k=0$ is impossible by Theorem \ref{conditions}\ref{4}. Thus, we have $N-k=1.$ According to assertion~\ref{4} of Theorem~\ref{conditions} this holds if and only if $N=2$ and $k=1.$ Therefore, by   Theorem~\ref{conditions}\ref{5} we obtain 
$$
\#\{j \mid a_j\,=\,1\} \geqslant 2.
$$
\noindent So the only smooth Fano variety in this case should be a curve 
$$
X \subset \mathbb{P}(1,1,a_2)
$$
 of degree $d$ such that $2+a_2>d.$ However, by Theorem~\ref{conditions}\ref{1} we have 
$$
2+a_2>d \geqslant 2a_2.
$$
 Thus,  $a_2=1$ and~$d=2,$ that is, $X$ is a conic in~\mbox{$\mathbb{P}^2.$} 

\end{proof}

\begin{Lemma}\label{lemmaalla=1}
Let $X \subset \mathbb{P}^N$ be a smooth Fano complete intersection of codimension~$k$ which does not lie in a hyperplane. Assume that $X$ is not an intersection of quadrics. Then for~\mbox{$l=\log_3(N-k+2)$} the inequality 
\begin{equation}\label{eg:lemmaalla=1maininequality}
\sum_{i=1}^k d_i^l \geqslant N+1
\end{equation}
\noindent holds true.
\end{Lemma}
\begin{proof}
Let $(d_1, \ldots, d_k)$ be the multidegree of $X \subset \mathbb{P}^N.$ By our assumptions there is $d_i$ such that~\mbox{$d_i>2.$} Suppose that the inequality~(\ref{eg:lemmaalla=1maininequality}) is false, that is,
$$
\sum_{i=1}^k d_i^l\,<\,N+1.
$$
\noindent Then we have 
$$
(k-1)2^l+3^l \leqslant \sum_{i=1}^k d_i^l\,<\,N+1.
$$
\noindent Thus, 
$$
(k-1)2^l+N-k+2\,<\,N+1,
$$
\noindent which is equivalent to $(k-1)\cdot2^l\,<\,k-1,$ but this is impossible because $l>1$ by Lemma \ref{lemmaaboutconic}.

\end{proof}

Now we are ready to prove Theorem~\ref{log_3}.
\begin{proof}[Proof of Theorem \ref{log_3}] Let $X \subset \mathbb{P}(a_0,a_1,\ldots,a_N)$ be a smooth well-formed weighted complete intersection of multidegree $(d_1, \ldots, d_k)$ which is not an intersection with a linear cone. Without loss of generality we can assume that
$$
d_1 \leqslant d_2 \leqslant \ldots  \leqslant d_k \qquad \text{and} \qquad a_0 \leqslant a_{1} \leqslant \ldots  \leqslant a_N.
$$

\noindent We can also assume that $X$ is a Fano variety, because otherwise there is nothing to prove. We also assume that $X$ is not a complete intersection of quadrics in $\mathbb{P}^N$ in view of Example~\ref{exquadric}. It is sufficient to prove that for~\mbox{$l = \lceil \log_3(N-k+2) \rceil$}, which is at least $2$ by Lemma \ref{lemmaaboutconic}, the $l$-th Chern character is nonpositive, that is,
\begin{equation}\label{eg:thlog_3maininequality}
\sum_{i=1}^k d_i^l \geqslant \sum_{j=0}^N a_j^l.
\end{equation}
Let $s\,=\,\#\{a_j \mid a_j\,=\,1\}.$  For~\mbox{$s=N+1$} Theorem~\ref{log_3} holds true by Lemma~\ref{lemmaalla=1}. So assume that $s \leqslant N.$ 

Suppose that $a_N=2.$ Then we have $a_j=2$ for all $s \leqslant j \leqslant N.$ Since 
$$
\#\{j \mid 2 \;\; \text{divides} \;\; a_j\}=N-s+1
$$
\noindent  by Theorem~\ref{conditions}\ref{7} we obtain $k \geqslant N-s+1.$ Hence the inequality (\ref{eg:thlog_3maininequality}) holds by Lemma \ref{lemmaallbut1a=1}.

 Otherwise, if $a_N \geqslant 3,$ the assumptions of Lemma~\ref{lemmaapB} hold by  
Theorem~\ref{conditions}. 
Thus, inequality (\ref{eg:thlog_3maininequality}) follows from this lemma.

\end{proof}

\section{Monotonicity and boundedness}\label{append}
In this section we prove some elementary lemmas about the monotonicity and boundedness of functions that we need in the proof of  Theorems~\mbox{\ref{log_3} and~\ref{monotonic}}.

\begin{Lemma}\label{LemmaAT}
If $a,$ $b$ and $c$ are integer numbers such that
 $$
a \geqslant b \geqslant c>0,
$$ 
 then the function $f(l)=a^l+b^l$ grows slower than the function $g(l)=(a+c)^l+(b-c)^l$ for $l \geqslant 1$ in $\mathbb{R};$ in other words, the function $g(l)-f(l)$ is monotonically increasing for $l \geqslant 1.$
\end{Lemma}

\begin{proof}
We have 
\begin{equation}
h(l)=g(l)-f(l)=(a+c)^l+(b-c)^l-a^l-b^l.
\end{equation}
\noindent If $b=c,$ then we are done because 
$$
h(l)=(a+b)^l-a^l-b^l=b^l\left( \left(\frac{a+b}{b}\right)^l-\left(\frac{a}{b}\right)^l-1\right)
$$
\noindent which is monotonically increasing as a product of monotonically increasing nonnegative functions. So assume that $b>c.$ We have
\begin{equation}
h(l)=(b-c)^l\left( \left(\frac{a+c}{b-c}\right)^l-\left(\frac{a}{b-c}\right)^l-\left(\frac{b}{b-c}\right)^l+1 \right).
\end{equation}

\noindent Put $y=a/(b-c)$ and $z=b/(b-c).$  Then we can rewrite $h(l)$ as
$$
h(l)=(b-c)^l\left( (y+z-1)^l-y^l-z^l+1\right)=(b-c)^lt(l,y,z).
$$

\noindent The function $(b-c)^l$ is monotonically increasing. So it is sufficient to prove that~$t(l,y,z)$ is monotonically increasing with respect to $l.$ Note that 
$$
y+z-1>y \geqslant z>1
$$
\noindent by our assumptions. We have 
\begin{equation}
s(l,y,z)=\frac{\partial t(l,y,z)}{\partial l}=\ln(y+z-1)\cdot(y+z-1)^l-\ln(y)\cdot y^l-\ln(z)\cdot z^l.
\end{equation}

To prove that $t(l,y,z)$ is monotonically increasing, we need to show that $s(l,y,z)$ is monotonically increasing with respect to $z$ and obtain
$$
s(l,y,z)>s(l,y,1)= 0.
$$
\noindent   In fact,
$$
\frac{\partial s(l,y,z)}{\partial z}=(y+z-1)^{l-1}-z^{l-1}+l\ln(y+z-1)\cdot(y+z-1)^{l-1}-l\ln(z)\cdot z^{l-1}>0.
$$
\noindent Thus, $h(l)$ is monotonically increasing for $l \geqslant 1.$

\end{proof}

\begin{Lemma}\label{lemmaAT2}
Let $a,$ $b$ and $d$ be real numbers such that $1 \leqslant d \leqslant a \leqslant b.$  Then 
$$
a+b \leqslant \frac{a}{d}+bd.
$$
\end{Lemma}

\begin{proof}
Indeed, we have 
$$
bd+\frac{a}{d}-a-b=\frac{a}{d}\left(\frac{bd^2}{a}+1-d-\frac{bd}{a} \right)=\frac{a}{d}\left( d-1 \right) \left( \frac{bd}{a}-1 \right) \geqslant 0.
$$

\end{proof}

\begin{Lemma}\label{4^l>2^l+3^l}
Let $f(l)\,=\,4^l-2^l-3^l,$ where $l \in \mathbb{R}.$ Then $f(l)\,>\,0$ for $l \geqslant 2.$ 
\end{Lemma}

\begin{proof}
Indeed, the function $f(l)/2^l$ is monotonically increasing for $l \geqslant 2$ because
$$
\frac{\partial}{\partial l}\left(\frac{f(l)}{2^l}\right)=2^l \ln2-\left(\frac{3}{2}\right)^l \ln\left(\frac{3}{2}\right)>\ln2 \left(2^l-\left(\frac{3}{2}\right)^l\right)>0.
$$

\noindent Moreover, we have $f(2)=3>0.$ Thus, $f(l)$ is positive for all $l \geqslant 2.$

\end{proof}

Now we prove the lemma used in the proof of Theorem \ref{monotonic}.

\begin{Lemma}\label{f'(l)>0}
Let  $N,$ $k,$ $d_1, \ldots, d_k$ and $a_0,$ $a_{1}, \ldots, a_N$  be positive integers such that
$$
d_1 \leqslant d_{2} \leqslant \ldots \leqslant d_{k-1}  \leqslant d_{k} 
$$
and
$$
 1=a_0=a_1= \ldots = a_{s-1}<a_s  \leqslant \ldots \leqslant a_{N-1} \leqslant a_N.
$$

\noindent  Suppose the following conditions are satisfied:
\begin{gather}
k \leqslant N+1;\\
d_i \geqslant 2 \quad \text{for all} \;\;  1 \leqslant i \leqslant k;\\  
d_i > a_{N-k+i} \quad \text{for all} \;\;  1 \leqslant i \leqslant k; \label{lemmad>a}\\
s=\#\{a_j\mid a_j=1\} \geqslant 1;\\
\sum_{j=s}^N a_j\, \leqslant \, \sum_{i=1}^k d_i \,<\,\sum_{j=s}^N a_j+s.\label{lemmasad}
\end{gather}

\noindent Let  
$$
f(l)\,=\,\sum_{i=1}^k d_i^l-\sum_{j=0}^N a_j^l,
$$

\noindent where $l \in \mathbb{R}.$ Then $f(l)$  is monotonically increasing for $l \geqslant 1.$ 
\end{Lemma}

\begin{proof} 
To prove the lemma we launch an algorithm such that at every step we change $a_j$, $d_i$, $N$, $k$ and $s$ by $a'_j$, $d'_i$, $N'$, $k'$ and $s'$, respectively, so that the new function
$$
f'(l)\,=\,\sum_{i=1}^{k'} (d_i')^l-\sum_{j=0}^{N'} (a_j')^l
$$
\noindent grows slower than $f(l).$ Dropping primes  for simplicity, we apply the algorithm again, until we obtain $s=N$ or $s=N+1$. Then we prove that after the algorithm stops we obtain a monotonically increasing function. Thus, the function $f(l)$ we have started with is also monotonically increasing. Let us describe a step of the algorithm.

We take $a_{N}$ and $a_s$ (that is, the minimum $a_j$ distinct from $1$), and we increase and decrease them, respectively, by the same integer $c$ so that 
\begin{gather}
a_s-c \geqslant 1;\label{a_s-c}\\
a_{N}+c \leqslant d_{k}.\label{a_N-1+c}
\end{gather}

\noindent  Let us take the maximum integer $c$ such that at least one of inequalities~\eqref{a_s-c} and~\eqref{a_N-1+c} is an equality and the other inequality is true, that is,
$$
c=\min\{a_s-1,d_k-a_N\}.
$$

 If (\ref{a_s-c}) is an equality and  (\ref{a_N-1+c}) is a strict inequality, we increase~$s$ by $1$, that is, put $s'=s+1$, $N'=N$ and $k'=k$ and change $a_j$ and $d_i$ to $a'_j$ and $d'_i,$ respectively, which are defined by
\begin{gather*}
a'_{s'-1}=1; \;\; a'_{s'}=a_{s+1}; \;\; a'_j=a_j \;\; \text{for} \;\;  j\neq s'-1, s',N'; \;\; a'_{N'}=a_{N}+c; \\
 d'_i=d_i \;\; \text{for} \;\; 1 \leqslant i \leqslant k'.
\end{gather*}

 If (\ref{a_s-c}) is a strict inequality and (\ref{a_N-1+c}) is an equality, then
we decrease $k$ and~$N$ by~$1$, that is, put  $k'=k-1$, $N'=N-1$ and  $s'=s$ and change $a_j$ and $d_i$ to $a'_j$ and $d'_i$, respectively, which are defined by
$$
a'_{s'}=a_s-c; \;\; a'_j=a_j \;\; \text{for} \;\;  j\neq s'; \;\ d'_i=d_i \;\; \text{for} \;\; 1 \leqslant i \leqslant k'.
$$

 If both (\ref{a_s-c}) and  (\ref{a_N-1+c}) are equalities, then we increase $s$ by $1$,  decrease $k$ and~$N$ by~$1$, that is,  put  $k'=k-1$,  $N'=N-1$ and $s'=s+1$ and  change $a_j$ and $d_i$ to~$a'_j$ and $d'_i$, respectively, which are defined by
$$
a'_{s'-1}=1; \;\; a'_{s'}=a_{s+1}; \;\; a'_j=a_j \;\; \text{for} \;\; s'+1 \leqslant j \leqslant N'; \;\; d'_i=d_i \;\; \text{for} \;\; 1 \leqslant i \leqslant k'.
$$

 By Lemma \ref{LemmaAT} the function $-(a_N+c)^l-(a_s-c)^l$ grows slower than $-a_N^l-a^l_s$. Thus, in each of the above three cases the new function
$$
f'(l)\,=\,\sum_{i=1}^{k'} (d_i')^l-\sum_{j=0}^{N'} (a_j')^l
$$
\noindent grows slower than $f(l).$ Also note that all the conditions of the lemma are satisfied for~$a'_j,$ $d'_i,$ $N',$ $k'$ and $s'.$

 The algorithm stops when we have either $s=N$, or $s=N+1$. Note that the case $s=N+1$ is possible if and only if either at the last step   both (\ref{a_s-c}) and  (\ref{a_N-1+c}) are equalities, or at the beginning of the algorithm $s=N+1.$ If $s=N+1$ at the beginning of the algorithm,  then the monotonicity of the function
$$
f(l)\,=\,\sum_{i=1}^k d_i^l-N-1
$$
\noindent is obvious. So assume that at the beginning of the algorithm $s \leqslant N.$  From (\ref{lemmasad}) we get that, after the algorithm stops, we have $k \geqslant 1.$ In this case  it suffices to prove that $f(l)=\sum^k_{i=1}d^l_i-a^l_N-s$ is monotonically increasing. We can rewrite this function as 
$$
f(l)=\left(d_k^l-a^l_{N}\right)+\sum^{k-1}_{i=1}d_i^l-s
$$
\noindent and get that $f(l)$ is monotonically increasing function as a sum of monotonically increasing functions. In the  case $s=N+1$  it suffices to prove that 
$$
f(l)=\sum^k_{i=1}d^l_i-N-1
$$
\noindent is monotonically increasing, which is obviously true.

\end{proof}

\section{Inequalities on sums of $l$th degrees}\label{appendfortheoremlog_3}

In this section we prove some computational lemmas  needed in Section \ref{prooflog_3}.

\begin{Lemma}\label{lemmaallbut1a=1} 
Let $N,$ $k,$ $a_0, a_1, \ldots, a_N$ and  $d_1, d_2, \ldots, d_k$ be positive integer numbers such that the following inequalities hold:
\begin{gather}
N-k \geqslant 2;\label{lemma5.3N-k>2}\\
1=a_0=a_{1}=\ldots =a_{s-1}<a_s \leqslant \ldots \leqslant a_{N-1} \leqslant a_N;\\
a_N \geqslant 2;\label{lemma5.3a>2}\\
d_k \geqslant 2a_N;\label{lemma5.3appd>2a}\\
d_i > a_{N-k+i}  \quad \text{for} \quad 1 \leqslant i \leqslant k-1.\label{lemma5.3d>a}
\end{gather}
\noindent Assume that $1 \leqslant \#\{a_j \mid a_j \neq 1\} \leqslant k.$ Then for $l=\lceil \log_3(N-k+2) \rceil$ the inequality 
\begin{equation}\label{eg:lemmaallbuta=1maininequality}
\sum_{i=1}^k d_i^l \geqslant \sum_{j=0}^N a_j^l
\end{equation}
\noindent holds true.
\end{Lemma}

\begin{proof}
Since $\#\{a_j \mid a_j \neq 1\}=N-s+1,$ we have $N-s+1 \leqslant k.$ We need to prove that
\begin{equation}\label{eg:lemmaallbuta=1maininequalitys}
\sum_{i=1}^k d_i^l -\sum_{j=s}^N a_j^l \geqslant s.
\end{equation}

\noindent By (\ref{lemma5.3appd>2a}) we have 
$$
\sum_{i=1}^k d_i^l -\sum_{j=s}^N a_j^l \geqslant a_N^l(2^l-1)+\sum_{i=1}^{k-1} d_i^l-\sum_{j=s}^{N-1}a_j^l.
$$
\noindent Since $l \geqslant 2$   by (\ref{lemma5.3N-k>2}), from \eqref{lemma5.3d>a} we obtain the inequalities 
\begin{equation}\label{eq:dl>al+1}
d_i^l>(a_{N-k+i})^l+1 \quad \text{for} \quad 1 \leqslant i \leqslant k-1.
\end{equation}

\noindent  Taking \eqref{eq:dl>al+1} and the inequality $N-s+1 \leqslant k$ into account  we obtain
\begin{equation}\label{eq:sumd-a}
\sum_{i=1}^{k-1} d_i^l-\sum_{j=s}^{N-1}a_j^l > \sum_{i=1}^{k-1} \left((a_{N-k+i})^l+1  \right)-\sum_{j=s}^{N-1}a_j^l \geqslant k-1,
\end{equation}
\noindent since 
$$
\sum_{i=1}^{k-1} a_{N-k+i}^l-\sum_{j=s}^{N-1}a_j^l = \sum_{m=N-k+1}^{N-1} a_{m}^l-\sum_{j=s}^{N-1}a_j^l \geqslant 0.
$$

\noindent Thus, using (\ref{lemma5.3a>2}) and (\ref{eq:sumd-a}), we have
$$
a_N^l(2^l-1)+\sum_{i=1}^{k-1} d_i^l-\sum_{j=s}^{N-1}a_j^l \geqslant 2^l(2^l-1)+k-1.
$$

\noindent From (\ref{lemma5.3N-k>2}) we obtain $l \geqslant 2.$ So  by Lemma \ref{4^l>2^l+3^l} and the definition of $l$ we have 
\begin{equation}\label{eq:3^l>N-k+2B}
2^l(2^l-1) \geqslant 3^l \geqslant N-k+2.
\end{equation}

\noindent Hence  we have  
\begin{equation}\label{eq:s-k+1}
2^l(2^l-1)+k-1 \geqslant N-k+2+k-1=N+1 \geqslant s,
\end{equation}

\noindent by (\ref{eq:3^l>N-k+2B}), which proves  (\ref{eg:lemmaallbuta=1maininequalitys}).

\end{proof}

\begin{Lemma}\label{lemmaallbut1a=1:case1} 
Let $N,$ $k$ and  $d_1,d_2,\ldots, d_k$ be positive integer numbers and $a_{N-1}$ and $a_N$ be positive rational numbers such that the following inequalities hold:
\begin{gather}
N-2 \geqslant k \geqslant 2;\label{lemma5.3N-k>21}\\
d_i \neq 1 \quad \text{for all} \quad i; \label{lemma5.3dneq1case1}\\
a_{N-1}>1; \label{lemma5.3an-1>1case1}\\
a_N \geqslant 2;\label{lemma5.3a>21}\\
d_k \geqslant 2a_N;\label{lemma5.3appd>2a1}\\
d_{k-1} > a_{N-1}.\label{lemma5.3dk-1>a1}
\end{gather}
\noindent Then for $l=\lceil \log_3(N-k+2) \rceil$ the inequality 
\begin{equation}\label{eg:lemmaallbuta=1maininequality1}
\sum_{i=1}^k d_i^l \geqslant  a_N^l+a_{N-1}^l+N-1
\end{equation}
\noindent holds true.
\end{Lemma}

\begin{proof}
 We need to prove that
\begin{equation}\label{eg:lemmaallbuta=1maininequalitys1}
\sum_{i=1}^k d_i^l -a_N^l-a_{N-1}^l \geqslant N-1.
\end{equation}

\noindent By (\ref{lemma5.3appd>2a1}) we have 
\begin{equation}\label{eq:d>2aininequality}
\sum_{i=1}^k d_i^l -a_N^l-a_{N-1}^l \geqslant a_N^l(2^l-1)+\sum_{i=1}^{k-2} d_i^l+\left(d^l_{k-1}-a_{N-1}^l\right).
\end{equation}

\noindent Taking \eqref{lemma5.3dneq1case1}, \eqref{lemma5.3a>21} and  \eqref{lemma5.3dk-1>a1}  into account, we obtain from \eqref{eq:d>2aininequality} the inequality
\begin{equation}
\sum_{i=1}^k d_i^l -a_N^l-a_{N-1}^l > 2^l(2^l-1)+2^l(k-2).
\end{equation}

\noindent By (\ref{lemma5.3N-k>21}) we have $l \geqslant 2.$ So  by Lemma \ref{4^l>2^l+3^l} and the definition of $l$ we get
\begin{equation}\label{eq:3^l>N-k+2B1}
2^l(2^l-1) \geqslant 3^l \geqslant N-k+2.
\end{equation}

\noindent Hence by (\ref{eq:3^l>N-k+2B1}) we have  
$$
2^l(2^l-1)+2^l(k-2) \geqslant N-k+2+2^l(k-2) \geqslant N+(2^l-1)(k-2) > N-1,
$$

\noindent which proves the inequality (\ref{eg:lemmaallbuta=1maininequalitys1}).

\end{proof}

\begin{Lemma}\label{case3B}
Let $d$ and $N \geqslant 3$ be  positive integer numbers and $a_0,a_1,\ldots, a_N$ be rational numbers. Assume that the following holds
\begin{gather}
1=a_0 = a_1 = \ldots = a_{s-1}<a_s \leqslant \ldots \leqslant a_{N-1} \leqslant a_N;\label{case3Ba0<aN}\\
d=m\cdot a_0 \cdot \ldots \cdot a_N \;\; \text{for some positive integer} \;\; m;\label{case3Bdivisib}\\
a_N \geqslant 3;\label{case3a>3}\\
d \geqslant 2a_N.\label{case3d>a}
\end{gather}
\noindent Assume also that all the $a_j,$ with a possible exception of $a_s,$ are integers and~\mbox{$s \leqslant N-1.$} If 
$
l=\lceil \log_3(N+1)\rceil,
$
\noindent then
\begin{equation}\label{case3Bmaininequality}
 d^l \geqslant \sum_{j=0}^N a_j^l.
\end{equation}
\end{Lemma}

\begin{proof}
 Since $N \geqslant 3,$ we have $l \geqslant 2.$ So first of all, assume that $d \geqslant 3a_N.$ We have to prove that 
\begin{equation}\label{case3Bk=1}
3^l a_N^l -\sum_{j=s}^{N}a_j^l \geqslant s.
\end{equation}
\noindent By (\ref{case3Ba0<aN}) we have
$$
\sum_{j=s}^N a_j^l \leqslant a_N^l(N-s+1).
$$
\noindent Hence
$$
3^l a_N^l -\sum_{j=s}^{N}a_j^l \geqslant 3^la_N^l-a_N^l(N-s+1).
$$
\noindent As $3^l \geqslant N+1,$ we obtain
$$
3^la_N^l-a_N^l(N-s+1) \geqslant sa_N^l \geqslant s,
$$
\noindent which proves the inequality (\ref{case3Bk=1}).

Now  assume that $3a_N > d \geqslant 2a_N.$ By (\ref{case3Bdivisib}) we have 
$$
d=a_N(ma_{N-1} \ldots a_s)=a_N\cdot T,
$$
\noindent where $m$ is a positive integer and $T=ma_{N-1} \ldots a_s.$  The number $T$ is at least~$3$, unless either 
$$
s=N-2, \;\;  a_{N-1}=2 \;\; \text{and} \;\; a_{N-2} < 2,
$$
 or 
$$
s=N-1 \;\; \text{and} \;\; a_{N-1}<3.
$$

First of all, consider the case when  $s=N-2$ and $a_{N-1}=2.$ By \eqref{case3Ba0<aN} we have the inequality $a_{N-2} \leqslant 2.$  So it suffices to prove that
$$
2^la_N^l \geqslant a_N^l+2^l+2^l+N-2,
$$
\noindent or, equivalently,
$$
(2^l-1)a_N^l-2^l-2^l-N+2 \geqslant 0.
$$

\noindent From (\ref{case3a>3}) we get that
$$
(2^l-1)a_N^l-2^l-2^l-N+2 \geqslant 3^l(2^l-1)-2\cdot 2^l-N+2.
$$
\noindent As $3^l \geqslant N+1,$ we have 
\begin{gather*}
3^l(2^l-1)-2\cdot 2^l-N+2 \geqslant (N+1)(2^l-1)-2\cdot 2^l-N+2=\\
=2^l(N-1)-2(N-1)-1 \geqslant 0
\end{gather*}

\noindent because $l \geqslant 2.$ 

Now consider the case when  $s=N-1$ and $a_{N-1}<3.$ It suffices to prove that
$$
2^la_N^l \geqslant a^l_N+3^l+N-1.
$$
\noindent In other words, as $a_N \geqslant 3,$  it suffices to prove that
$$
3^l(2^l-2) \geqslant N-1,
$$
\noindent which is true because $l \geqslant 2$ and $3^l \geqslant N+1.$
\end{proof}

\begin{Lemma}\label{lemmaapB}
Let $N,$ $k$ and $d_1, d_2, \ldots, d_k$ be positive integer numbers and $a_0, a_1, \ldots, a_N$ be rational numbers  such that the following inequalities hold
\begin{gather}
N-k \geqslant 2;\label{appBN-k>2}\\
 d_1 \leqslant d_{2} \leqslant \ldots \leqslant d_{k-1}  \leqslant d_{k}; \label{eq:enequalityofdiappB}\\
1=a_0 = a_1 = \ldots = a_{s-1}<a_s \leqslant \ldots \leqslant a_{N-1} \;\; \text{and} \;\; a_{N-2} \leqslant a_N;\label{appBa0<aN}\\
d_i \,>\, a_{N-k+i} \quad \text{for} \quad 1 \leqslant i \leqslant k-1;\label{appBd>a}\\
d_i \neq 1 \quad \text{for all} \quad i;\label{appBd>1}\\
a_N \geqslant 3;\label{appBa>3}\\
d_1\cdot \ldots \cdot d_k=m\cdot a_0 \cdot \ldots \cdot a_N \;\; \text{for some positive integer} \;\; m;\label{appBdivisib2}\\
d_k \geqslant 2a_N.\label{appBd>2a}
\end{gather}
 \noindent Assume that all the $a_j,$ with a possible exception of $a_s$ and $a_{N-1},$ are integer numbers.  If~$l=\lceil \log_3(N-k+2)\rceil,$ then 
\begin{equation}\label{appBmaininequality}
\sum_{i=1}^k d_i^l \geqslant \sum_{j=0}^N a_j^l.
\end{equation}
\end{Lemma}

\begin{proof} 

To prove the lemma we launch an algorithm such that at every step we change the numbers $a_j$, $d_i$, $N$, $k$ and $s$ to $a'_j$, $d'_i$, $N'$, $k'$ and $s'$, respectively, so that
$$
\sum_{i=1}^{k} d_i^l - \sum_{j=0}^{N} a_j^l \geqslant \sum_{i=1}^{k'}(d'_i)^l - \sum_{j=0}^{N'}  (a'_j)^l.
$$
\noindent Then dropping  primes  for simplicity we apply the algorithm again, until we obtain~$a_j,$~$d_i,$~$N,$~$k$ and $s$ satisfying certain strong restrictions. We prove that after the algorithm stops we obtain a nonnegative quantity
\begin{equation}\label{eq:strange}
\sum_{i=1}^{k} d_i^l - \sum_{j=0}^{N} a_j^l \geqslant 0.
\end{equation}

\noindent Thus, the quantity we have started with is also nonnegative. Let us describe the step of the algorithm. First of all, note that the algorithm stops in one of the following cases:
\begin{itemize}
\item[(1)]
$s=N-1$ and $k>1$;
\item[(2)]
$s=N$ for an arbitrary $k$;
\item[(3)]
$k=1$ and $s \leqslant N-1$. 
\end{itemize}
\noindent  So we take $a_{N-1}$ and $a_s$ and multiply and divide them, respectively, by the same rational number $p$ with the properties 
\begin{gather}
\frac{a_s}{p} \geqslant 1;\label{eq:algforas}\\
a_{N-1}\cdot p \leqslant d_{k-1}.\label{eq:algforaN-1}
\end{gather}

  Let us take the maximum $p$ such that at least one of  inequalities (\ref{eq:algforas}) and~(\ref{eq:algforaN-1}) is an equality and the other inequality is true, that is,
$$
p=\min\{a_s,\frac{d_{k-1}}{a_{N-1}}\}.
$$

 If (\ref{eq:algforas}) is an equality and (\ref{eq:algforaN-1}) is a strict inequality, we increase $s$ by~$1$, that is, put $s'=s+1$, $N'=N$ and $k'=k$ and change $a_j$ and $d_i$ to~$a'_j$ and $d'_i,$ respectively, which are defined by
\begin{gather*}
a'_{s'-1}=1, \;\; a'_{s'}=a_{s+1}, \;\; a'_{j}=a_j \;\; \text{for} \;\;  j \neq s'-1,s',N'-1, \\
a'_{N'-1}=p \cdot a_{N-1}, \;\; d'_i=d_i \;\; \text{for} \;\; 1 \leqslant i \leqslant k'.
\end{gather*}
\noindent After a step of the algorithm we have 
\begin{equation}\label{eq:N'-k'+2first}
N-k+2=N'-k'+2.
\end{equation}
\noindent Let us show that  the assumptions  of the lemma still hold true. Indeed, condition~\eqref{appBN-k>2} follows from \eqref{eq:N'-k'+2first}, condition~\eqref{eq:enequalityofdiappB} is obvious,  and condition~\eqref{appBa0<aN} follows from the  chain of  inequalities
$$
1=a'_0 = a'_1 = \ldots = a'_{s'-1}<a'_{s'}=a_{s+1} \leqslant \ldots \leqslant a'_{N'-2}=a_{N-2} \leqslant a'_{N'-1}=p \cdot a_{N-1}
$$
\noindent and the relation
$$
a'_{N'-2}=a_{N-2} \leqslant a_N=a'_{N'}.
$$

\noindent Condition~\eqref{appBd>a} follows from the inequality
$$
d'_i=d_i \,>\, a_{N-k+i}=a'_{N-k+i}=a'_{N'-k'+i}
$$
\noindent for 
$$
N'-k'+1 \leqslant N'-k'+i \leqslant N'-2 \quad \text{and} \quad i \neq s'-1-N'+k', s'-N'+k', 
$$
\noindent  from the inequality
$$
d'_{s'-1-N'+k}>1=a'_{s'-1},
$$
which follows  from \eqref{appBd>1} if $s'-1-N'+k >0,$ from the inequality
$$
d'_{s'-N'+k'}=d_{s+1-N+k}>a_{s+1}=a'_{s'}
$$
\noindent if $s'-N'+k >0,$  and from the inequality
$$
d'_{k'-1}=d_{k-1} > p \cdot a_{N-1}=a'_{N'-1}
$$
\noindent which follows from \eqref{eq:algforaN-1}. Condition \eqref{appBd>1} is obvious. Condition~\eqref{appBa>3} holds because $a'_{N'}=a_N,$ condition~\eqref{appBdivisib2} follows from the equalities
\begin{gather*}
d'_1\cdot \ldots \cdot d'_{k'}=d_1\cdot \ldots \cdot d_k=m\cdot a_0 \cdot \ldots \cdot a_s \cdot \ldots \cdot a_{N-1} \cdot a_N=\\
=m\cdot a'_0 \cdot \ldots \cdot p  \cdot \ldots \cdot \frac{a'_{N'-1}}{p} \cdot a'_{N'}=m\cdot a'_0 \cdot  \ldots \cdot a'_{N'-1} \cdot a'_{N'}.
\end{gather*}
\noindent Condition~\eqref{appBd>2a} follows form the inequality
$$
d'_{k'}=d_k \geqslant 2a_N=2a'_{N'}.
$$

\enskip

If (\ref{eq:algforas}) is a strict inequality and (\ref{eq:algforaN-1}) is an equality, we decrease~$k$ and $N$ by~$1$, that is,
$$
k'=k-1, \;\; N'=N-1 \;\; \text{and} \;\; s'=s,
$$
\noindent and change $a_j$ and $d_i$ to $a'_j$ and $d'_i,$ respectively, which are defined by
\begin{gather*}
a'_{s'}=\frac{a_s}{p}, \;\; a'_j=a_j \;\; \text{for}\;\; j \neq s',N', \;\; a'_{N'}=a_N, \\
  d'_i=d_i \;\; \text{for} \;\; 1 \leqslant i \leqslant k'-1 \;\; \text{and} \;\; d'_{k'}=d_k .
\end{gather*}
\noindent After a step of the algorithm we have 
\begin{equation}\label{eq:N'-k'+2second}
N-k+2=N'-k'+2.
\end{equation}
\noindent Let us show that  the assumptions  of the lemma still hold true.  Indeed, condition~\eqref{appBN-k>2} follows from \eqref{eq:N'-k'+2second}, condition~\eqref{eq:enequalityofdiappB} is obvious,  and condition~\eqref{appBa0<aN} follows from the  chain of  inequalities
$$
1=a'_0 = a'_1 = \ldots = a'_{s'-1}<a'_{s'}=\frac{a_{s}}{p} \leqslant a'_{s'+1}=a_s \leqslant \ldots  \leqslant a'_{N'-1}=a_{N-2},
$$
\noindent which holds since $1<p<a_s$ by~\eqref{appBd>a}, and from the  inequality
$$
a'_{N'-2}=a_{N-3} \leqslant a_{N-2} \leqslant a_N=a'_{N'}.
$$

\noindent Condition~\eqref{appBd>a} follows from the inequality
$$
d'_i=d_i \,>\, a_{N-k+i}=a'_{N'-k'+i}=a'_{N-k+i}
$$
\noindent for 
$$
N'-k' \leqslant N'-k'+i \leqslant N'-1 \quad \text{and} \quad i \neq s'-N'+k',
$$
\noindent and from the inequality 
$$
d'_{s'-N'+k'}=d_{s-N+k}>a_s>\frac{a_s}{p}=a'_s
$$
\noindent which holds for $s'-N'+k'>0$  since $p>1.$ Condition \eqref{appBd>1} is obvious. Condition~\eqref{appBa>3} follows from the equality $a'_{N'}=a_N,$ and condition~\eqref{appBdivisib2} follows from the relations
\begin{gather*}
d'_1\cdot \ldots \cdot d'_{k'-1} \cdot d'_{k'}=d_1\cdot \ldots \cdot d_{k-2} \cdot d_k=\frac{m}{d_{k-1}}\cdot a_0 \cdot \ldots \cdot a_s \cdot \ldots \cdot a_{N-1} \cdot a_N=\\
=m \cdot a_0 \cdot \ldots \cdot \frac{a_{N-1}}{d_{k-1}} \cdot a_N=m a_0 \cdot \ldots \cdot \frac{1}{p} \cdot a_N=\\
=m \cdot a'_0 \cdot \ldots \cdot a'_{s'-1} \cdot \frac{a_s}{p} \cdot \ldots \cdot a'_{N'-1} \cdot a'_{N'}=\\
=m \cdot a'_0 \cdot \ldots \cdot a'_{s'-1} \cdot a'_{s'} \cdot  \ldots \cdot a'_{N'-1} \cdot a'_{N'}
\end{gather*}
\noindent Condition~\eqref{appBd>2a} follows form the inequality
$$
d'_{k'}=d_k \geqslant 2a_N=2a'_{N'}.
$$
\noindent Moreover, in additional, we have   strengthened assumption (\ref{appBa0<aN}):
$$
 a'_0 \leqslant a'_{1} \leqslant \ldots \leqslant a'_{N'-1} \leqslant a'_{N'},
$$
\noindent since 
$$
a'_{N'-1}=a_{N-2} \leqslant a_N=a'_{N'}
$$
\noindent and $a'_{N'-1}$ is an integer number unless $s'=N'-1.$

\enskip

If both (\ref{eq:algforas}) and  (\ref{eq:algforaN-1}) are equalities, we increase $s$ by $1$, decrease $k$ and $N$ by~$1$, that is,  $k'=k-1$, $N'=N-1$ and $s'=s+1$ and change $a_j$ and $d_i$ to $a'_j$ and $d'_i,$ respectively, which are defined by
\begin{gather*}
a'_{s'-1}=1, \;\; a'_{s'}=a_{s+1}, \;\; a'_{j}=a_j \;\; \text{for} \;\;  j \neq s'-1,s',N', \;\; a'_{N'}=a_{N},\\
 d'_i=d_i \;\; \text{for} \;\; 1 \leqslant i \leqslant k'-1 \;\; \text{and} \;\; d'_{k'}=d_k.
\end{gather*}
\noindent After a step of the algorithm we have 
\begin{equation}\label{eq:N'-k'+2third}
N-k+2=N'-k'+2.
\end{equation}
\noindent Let us show that  the assumptions  of the lemma still hold true.  Indeed, condition~\eqref{appBN-k>2} follows from \eqref{eq:N'-k'+2third}, condition~\eqref{eq:enequalityofdiappB} is obvious and  condition~\eqref{appBa0<aN} follows from the chain of  inequalities
$$
1=a'_0 = a'_1 = \ldots = a'_{s'-1}<a'_{s'}=a_{s+1} \leqslant  a'_{s'+1} \leqslant \ldots  \leqslant a'_{N'-1}=a_{N-2}
$$
\noindent  and  the  inequality
$$
a'_{N'-2}=a_{N-3} \leqslant a_{N-2} \leqslant a_N=a'_{N'}.
$$

\noindent Condition~\eqref{appBd>a} follows from the inequality
$$
d'_i=d_i \,>\, a_{N-k+i}=a'_{N-k+i}=a'_{N'-k'+i}
$$
\noindent for 
$$
N'-k' \leqslant N'-k'+i \leqslant N'-1 \quad \text{and} \quad i \neq s'-N'+k', s'-1-N'+k',
$$
\noindent  from the inequality
$$
d'_{s'-N'+k'}=d_{s+1-N+k}>a_{s+1}=a'_{s'}
$$
\noindent in the case $s'-N'+k'>0,$ or from the inequality
$$
d'_{s'-1-N'+k'}=d_{s-N+k}>a_{s}>1=a'_{s'-1}
$$
\noindent in the case when $s'-1-N'+k'>0.$ Consition \eqref{appBd>1} is obvious. Condition~\eqref{appBa>3} follows from the equality $a'_{N'}=a_N$ and condition~\eqref{appBdivisib2} follows from the relations
\begin{gather*}
d'_1\cdot \ldots \cdot d'_{k'-1} \cdot d'_{k'}=d_1\cdot \ldots \cdot d_{k-2} \cdot d_k=\frac{m}{d_{k-1}}\cdot a_0 \cdot \ldots \cdot a_s \cdot \ldots \cdot a_{N-1} \cdot a_N=\\
=m \cdot a_0 \cdot \ldots \cdot \frac{a_{N-1}}{d_{k-1}} \cdot a_N=m a_0 \cdot \ldots \cdot \frac{a_s}{p} \cdot \ldots \cdot a_{N-2} \cdot a_N=\\
=m \cdot a'_0 \cdot \ldots \cdot a'_{s'-1} \cdot a'_{s'} \cdot  \ldots \cdot a'_{N'-1} \cdot a'_{N'}.
\end{gather*}
\noindent Condition~\eqref{appBd>2a} follows from the inequality
$$
d'_{k'}=d_k \geqslant 2a_N=2a'_{N'}.
$$
\noindent Moreover, we have strengthened  assumption (\ref{appBa0<aN}):
\begin{equation}\label{appBa<aNstrong}
 a'_0 \leqslant a'_{1} \leqslant \ldots \leqslant a'_{N'-1} \leqslant a'_{N'},
\end{equation}
\noindent since 
$$
a'_{N'-1}=a_{N-2} \leqslant a_N=a'_{N'},
$$
\noindent and $a'_{N'-1}$ is an integer number, unless $s'=N'-1.$

In any case, by Lemma \ref{lemmaAT2} we have 
$$
\left(a_{N-1}\cdot p\right)^l+\left(\frac{a_s}{p}\right)^l \geqslant a_{N-1}^l+a_s^l.
$$ 
\noindent This implies that
$$
\sum_{i=1}^{k'}(d'_i)^l - \sum_{j=0}^{N'}  (a'_j)^l   \leqslant  \sum_{i=1}^{k} d_i^l - \sum_{j=0}^{N} a_j^l. 
$$

\noindent Note also that the algorithm does not change $a_N$ and $d_k.$ 

Recall that the algorithm stops if one of the following cases occurs:

\begin{itemize}
\item[(1)]
$s=N-1$ and $k>1$;
\item[(2)]
$s=N$ for arbitrary $k$;
\item[(3)]
$k=1$ and $s \leqslant N-1$. 
\end{itemize}

\noindent Let us prove that in all these cases  inequality (\ref{appBmaininequality}) holds true. In the first case we are done by Lemma~\ref{lemmaallbut1a=1:case1}. In the second case we are done by Lemma \ref{lemmaallbut1a=1}, because~$a_N$ is an integer.  The third case holds either in the beginning of the algorithm, or if the equality $a_{N-1}\cdot p =  d_{k-1}$ holds at its last step of the algorithm, because otherwise we have~\mbox{$s=N$} at the next to the last step,  and this is the second case. Thus,  inequalities  \eqref{appBa<aNstrong} hold and $a_{N-1}$ is an integer unless~\mbox{$s=N-1.$} So in this case we are done  by Lemma \ref{case3B}.

\end{proof}

%
%

\begin{thebibliography}{10}

\bibitem{l-Fano}
C.~Araujo, R.~Beheshti, A.-M.~Castravet, K.~Jabbusch, S.~Makarova, E.~Mazzon, L.~Taylor, N.~Viswanathan \emph{Higher Fano manifolds}, Rev. Un. Mat. Argentina \textbf{64} (2022), no. 1, 103--125.

\bibitem{Starr}
A.~J.~de Jong, J.~M.~Starr, \emph{Higher Fano manifolds and rational surfaces}, Duke Math. J. \textbf{139} (2007), no. 1, 173--183.

\bibitem{Starranote}
A.~J.~de Jong, J.~M.~Starr, \emph{A note on Fano manifolds whose second Chern
character is positive}, arXiv:math/0602644 [math.AG].

\bibitem{Kollar}
J.~Koll\'ar, \emph{Rational curves on algebraic varieties}, Springer-Verlag (1996).


\bibitem{Castravet2012}
C.~Araujo, A.-M.~Castravet, \emph{Polarized minimal families of rational curves and higher Fano manifolds}, Amer. J. Math. \textbf{134} (2012), no. 1, 87--107. 

\bibitem{Lang}
S.~Lang, \emph{On quasi algebraic closure}, Ann. of Math., \textbf{55} (1952), no. 2,  373--390.


\bibitem{Sato20}
Y.~Sano, H.~Sato, Yu.~Suyama, \emph{Toric Fano manifolds of dimension at most eight with positive second Chern characters}, Kumamoto J. Math. \textbf{34} (2021), 1--13.


\bibitem{IF}
A.~R.~Iano-Fletcher, \emph{Working with weighted complete intersections}, Explicit birational geometry of 3-folds, 101--173, London Math. Soc. Lecture Note Ser., \textbf{281}, Cambridge Univ. Press,
Cambridge, 2000.


\bibitem{PS20}
V.V.~Przyjalkowski, C.A.~Shramov, \emph{Hodge level for weighted complete intersections}, Collect. Math. \textbf{71} (2020), 549--574. 


\bibitem{Tasin}
M.~Pizzato, T.~Sano, L.~Tasin, \emph{Effective nonvanishing for Fano weighted complete intersections},
Algebra Number Theory \textbf{11} (2017), no. 10, 2369--2395.

\bibitem{CCC}
J.-J.~Chen, J.~Chen, M.~Chen, \emph{On quasismooth weighted complete intersections}, J. Algebraic
Geom. \textbf{20} (2011), no. 2, 239--262.	

\bibitem{PS16}
V.V.~Przyjalkowski, C.A.~Shramov, \emph{Bounds for smooth Fano weighted complete intersections}, Commun. Number Theory Phys., \textbf{14} (2020), no. 3, 511--553.

\bibitem{PS26}
V.V.~Przyjalkowski, C.A.~Shramov, \emph{Weighted complete intersections}, De Gruyter Exp. Math., \textbf{79}, De Gruyter, Berlin, Boston (2026). 

\bibitem{Cox}
D.~A.~Cox, J.~B.~Little, H.~K.~Schenck, \emph{Toric varieties}, Graduate Studies in Mathematics, \textbf{124} (2011).
	
	
\bibitem{Coskun}
I.~Coskun, J.~Lesieutre, J.~Ottem, \emph{Effective cones of cycles on blowups of projective space}, 
Algebra Number Theory \textbf{10} (2016), no. 9, 1983--2014.	
	
	

\bibitem{Sato19}
H.~Sato, Yu.~Suyama, \emph{Remarks on toric manifolds whose Chern characters are positive}, Commun. Algebra (2020).	
	
\bibitem{Sato18}
H.~Sato, Yu.~Suyama, \emph{Examples of singular toric varieties with certain numerical conditions}, Osaka J. Math.
\textbf{57} (2020), no. 1, 51–59.








\end{thebibliography}
%

\providecommand{\bysame}{\leavevmode\hbox to3em{\hrulefill}\thinspace}
\providecommand{\MR}{\relax\ifhmode\unskip\space\fi MR }
\providecommand{\MRhref}[2]{%
  \href{http://www.ams.org/mathscinet-getitem?mr=#1}{#2}
}
\providecommand{\href}[2]{#2}

\end{document}